\documentclass[twoside,reqno,11pt]{amsart}

\usepackage{amssymb}
\usepackage[mathscr]{eucal}

\usepackage{hyperref}

\numberwithin{equation}{section}

\theoremstyle{plain}
\newtheorem{prop}{Proposition}[section]
\newtheorem{thm}[prop]{Theorem}
\newtheorem{cor}[prop]{Corollary}
\newtheorem{lem}[prop]{Lemma}

\theoremstyle{definition}
\newtheorem{dfn}[prop]{Definition}
\newtheorem{lab}[prop]{}

\theoremstyle{remark}
\newtheorem{example}[prop]{Example}
\newtheorem{rem}[prop]{Remark}
\newtheorem{rems}[prop]{Remarks}

\renewcommand{\iff}{\Leftrightarrow}

\newcommand{\C}{{\mathbb{C}}}
\newcommand{\R}{{\mathbb{R}}}
\newcommand{\Z}{{\mathbb{Z}}}

\DeclareMathOperator{\disc}{disc}
\DeclareMathOperator{\res}{res}

\newcommand{\ol}{\overline}

\renewcommand{\epsilon}{\varepsilon}
\renewcommand{\theta}{\vartheta}
\def\lqf{\mathopen{<}}
\def\rqf{\mathclose{>}}

\renewcommand{\choose}[2]{\genfrac(){0pt}{}{#1}{#2}}

\begin{document}

\title
  [An elementary proof of Hilbert's theorem]
  {An elementary proof of\\Hilbert's theorem on ternary quartics}

\subjclass[2000]
{Primary
11E20, 
secondary
12D15} 


\author
  {Albrecht Pfister}
\address
  {Institut f\"ur Mathematik \\
  Johannes Gutenberg Universit\"at \\
  Staudingerweg 9 \\
  55099 Mainz \\
  Germany}
\email
  {pfister@mathematik.uni-mainz.de}

\author
  {Claus Scheiderer}
\address
  {Fachbereich Mathematik und Statistik \\
  Universit\"at Konstanz \\
  78457 Konstanz \\
  Germany}
\email
  {claus.scheiderer@uni-konstanz.de}

\begin{abstract}
In 1888, Hilbert proved that every nonnegative quartic form $f=f(x,y,
z)$ with real coefficients is a sum of three squares of quadratic
forms. His proof was ahead of its time and used advanced methods from
topology and algebraic geometry. Up to now, no elementary proof is
known. Here we present a completely new approach. Although our proof
is not easy, it uses only elementary techniques. As a by-product, it
gives information on the number of representations $f=p_1^2+p_2^2+p_3
^2$ of $f$ up to orthogonal equivalence. We show that this number is
$8$ for generically chosen $f$, and that it is $4$ when $f$ is chosen
generically with a real zero. Although these facts were known, there
was no elementary approach to them so far.
\end{abstract}

\maketitle


\section*{Introduction}

In 1888, David Hilbert published an influential paper \cite{H} which
became fundamental for real algebraic geometry, and which remains an
inspiring source for research even today.
It addresses the problem whether a real form (homogeneous polynomial)
$f(x_0,\dots,x_n)$ which takes nonnegative values on all of $\R^
{n+1}$ is necessarily a sum of squares of real forms. Hilbert proves
that the answer is negative in general. As is well-known, his results
go much beyond this fact and contain a surprising positive aspect as
well. Namely, for any pair $(n,d)$ of integers with $n\ge2$ and even
$d\ge4$, except for $(n,d)=(2,4)$, he shows that there exists a
nonnegative form of degree $d$ in $n+1$ variables which is not a sum
of squares of polynomials. In the exceptional case, however, he
proves that every nonnegative ternary quartic form is a sum of three
squares of real quadratic forms.

It is the existence of a representation $f=p_1^2+p_2^2+p_3^2$ in this
exceptional case that is the subject of the present article.
Hilbert's original proof is brief and elegant, and it is ahead of its
time in its topological arguments. For his contemporaries it must have
been hard to grasp.
Even today it is not easy to read, and it leaves a number of details
to be filled in. Several authors have given fully detailed accounts
of Hilbert's proof in recent years. We mention the approach due to
Cassels, published in Rajwade's book (\cite{Ra} chapter~7), and the
two articles by Rudin \cite{Ru} and Swan \cite{Sw}. These approaches
also show some characteristic differences.

One of the first approaches to Hilbert's theorem along elementary and
explicit lines was carried out by Powers and Reznick in \cite{PR},
where complete answers were given in certain special cases. We would
also like to point out the recent preprint \cite{PSV} by Plaumann,
Sturmfels and Vinzant which studies the computational side of
Hilbert's theorem, and which contains a beautiful blend of the 19th
century mathematics of ternary quartics.

So far, there seems to exist essentially only one proof different
from Hilbert's. It comes out as a by-product of the quantitative
analysis made in \cite{PRSS} and \cite{sch:jag}. These papers had a
different goal, namely to count the number of essentially distinct
ways in which a positive semidefinite (or \emph{psd}, for short)
ternary quartic $f$ can be written as a sum of three squares. The
case where the plane projective curve $f=0$ is non-singular is done
in \cite{PRSS}, the general irreducible case is in \cite{sch:jag}.
Both papers, and in particular the second, are using tools of modern
algebraic geometry and can certainly not be called elementary.

We are convinced that Hilbert's original proof from \cite{H} cannot
claim an elementary character either. This can be seen from the
following sketchy overview of its main steps:
\begin{itemize}
\item[(1)]
The set of sums of three squares of quadratic forms is closed inside
the space of all quartic forms. Therefore it suffices to prove the
existence of a representation for all forms in some open dense subset
of the psd forms, for example for all nonsingular such forms.
\item[(2)]
Hilbert proves that the map $(p_1,p_2,p_3)\mapsto\sum_{j=1}^3p_j^2$
(from triples of real quadratic forms to quartic forms) is submersive
(that is, its tangent maps are surjective), when restricted to the
open set of triples for which the curve $\sum_jp_j^2=0$ is
nonsingular. His elegant argument needs some non-trivial tool from
algebraic geometry, like Max Noether's $AF+BG$ theorem.
\item[(3)]
When the real form $f$ is strictly positive definite and singular,
the curve $f=0$ has at least two different (complex conjugate)
singular points.
\item[(4)]
The locus of quartic forms $f$ for which the curve $f=0$ has at least
two different singularities has codimension $\ge2$ inside the space
of all quartic forms.
\item[(5)]
Removing a subspace of codimension $\ge2$ from a connected
topological space leaves the remaining space connected. Hence, by (3)
and (4), the space of nonsingular positive forms is (path) connected.
\item[(6)]
There exist nonsingular positive forms which are sums of three
squares, like $f^{(0)}=x^4+y^4+z^4$.
\item[(7)]
Given an arbitrary nonsingular positive form $f$ there exists, by
(5), a path $f^{(t)}$, $0\le t\le1$, joining $f^{(1)}=f$ to a sum of
three squares $f^{(0)}$ such that $f^{(t)}$ is nonsingular and
positive for every $0\le t\le1$.
\item[(8)]
Using (1) and (2), and by the implicit function theorem, the
representation (6) of $f^{(0)}$ can be extended continuously along
the path $f^{(t)}$ to a representation of $f^{(1)}=f$ as a sum of
three squares.
\end{itemize}
In view of (2), and certainly of (4) and (5), this proof does not
have an elementary character. Also note that the existence of a path
$f^{(t)}$ as in (7) is ensured only by the general topological fact
(5). There is no concrete construction of such a path.

Our proof uses a variant of (1),
plus applications of the
implicit function theorem similar to (8). Otherwise it proceeds
differently. In particular, we avoid the non-elementary steps (2),
(4) and (5). Like Hilbert we are deforming representations along
paths. Other than in Hilbert's proof, however, our paths are
completely explicit, and are in fact simply straight line segments.
Here is a road map:
\begin{itemize}
\item[(a)]
By a limit argument (see \ref{limarg}), it suffices to prove the
existence of a representation for generic psd $f$, i.e., for psd $f$
satisfying a condition $\Psi(f)\ne0$ where $\Psi$ is a suitable
nonzero polynomial in the coefficients of $f$.
\item[(b)]
When the form $f$ has a non-trivial real zero, an elementary and
constructive proof for the existence of a representation as a sum of
three squares was given by the first author in \cite{Pf}. We shall
recall it in Sect.~\ref{elemreal0} below.
\item[(c)]
Assume that $f$ has no non-trivial real zero. We find a psd form
$f^{(0)}$ that has a non-trivial real zero such that the half-open
interval $\left]f^{(0)},f\right]$ (in the space of all quartic forms)
consists of strictly positive forms.
\item[(d)]
Let $f^{(t)}$ ($0\le t\le1$) denote the forms in the line segment
constructed in (c), with $f^{(1)}=f$. Under generic assumptions on
$f$ we show that every representation of $f^{(0)}$ can be extended
continuously to a representation of $f^{(t)}$ for $0<t<\epsilon$,
with some $\epsilon>0$.
\item[(e)]
Under further generic assumptions on $f$
we prove for every fixed $0<t\le1$ that every representation of $f^
{(t)}$ can be extended continuously and uniquely to a representation
of $f^{(s)}$ for all $s$ sufficiently close to $t$. Both in (d) and
(e) we use the theorem on implicit functions.
\item[(f)]
Using the limit principle (a), it follows that $f=f^{(1)}$ has a
representation as a sum of three squares.
\end{itemize}
All our ``generic assumptions'' on $f$ are explicit. See
\ref{exclist} for the entire list and for a discussion of where they
have been used. The exceptional cases that we have to exclude are
given by the vanishing of invariants that are mostly discriminants or
resultants of polynomials formed from (the coefficients of) $f$. Two
of our invariants are of a more general nature, one of them having
the amazing degree of~$896$ in the coefficients of $f$.

We believe that we have thus achieved a proof to Hilbert's theorem
that only uses elementary tools. With only little extra effort, our
arguments allow in fact to deduce substantial information on the
number of essentially distinct representations, at least in generic
cases. So far there has been no elementary approach to counting
representations. Therefore we think it worthwile to include these
parts.

Here is an overview of the structure of the paper. We start with the
case where $f$ has a real zero. By an explicit argument we show that
$f$ has a representation as a sum of three squares (Prop.\
\ref{exrep}). Refining the arguments yields the precise number of
inequivalent representations, under suitable hypotheses of generic
nature (Prop.\ \ref{summreal0}). In Section~\ref{sec:noreal0} we turn
to arbitrary psd quartic forms $f$. We show that $f$ can be written
as a sum of three squares, if and only if there exists a
polynomial-valued rational point (with certain side conditions) on a
certain elliptic curve associated with $f$ (Prop.\ \ref{pfthm}). No
background or terminology on elliptic curves is used. Again we refine
this by a result that permits to count representations (Prop.\
\ref{eindtfall1}). Then we construct the linear path $f^{(t)}$ ($0\le
t\le1$) referred to in (d) above and study the extension of
representations along this path. Extension around $t=0$ is studied in
Section~\ref{sec:defqui}, around $0<t<1$ in Sections
\ref{sec:defquii} and \ref{sec:defquiii}. In between we insert two
sections that provide the required background on symmetric functions.
Section~\ref{sec:disc} has classical material on the discriminant. To
handle the last case of the extension argument, we need an invariant
$\Phi(f,g,h)$ of triples of polynomials which is less standard; it is
introduced and discussed in Section \ref{sec:monster}. This invariant
essentially decides if the pencil spanned by $g$ and $h$ contains a
member that has a quadratic factor in common with $f$. We do not know
whether this invariant has been considered before. Finally, in
Section~\ref{sec:summ} we summarize our proof and give a systematic
account of all the genericity conditions used. We also obtain the
precise number of representations of $f$ under (explicit) generic
assumptions on $f$.

Basically, we consider techniques as ``elementary'' if they are
accessible using undergraduate mathematics. The most advanced
features that we use are the theorem on implicit functions and the
theorem on symmetric functions. Only once (in the proof of Prop.\
\ref{pfprop}(b)) are we using slightly more advanced algebraic
techniques, namely basic facts about Dedekind domains. However, this
part is only used for counting representations, and is not needed for
the proof of Hilbert's theorem.

We believe that our approach to representations as sums of three
squares is also ``constructive'', at least in a weak sense. It should
be possible to follow our deformation argument for constructing such
representations with arbitrary numeric precision, for example by
using finite element methods.


\section{The forms $\lqf1,q\rqf$}

As usual, a polynomial $f(x_1,\dots,x_n)$ with real coefficients is
said to be \emph{positive semidefinite} (or \emph{psd} for short) if
$f$ takes nonnegative values on $\R^n$. It is said to be
\emph{positive definite} if $f(x)>0$ for all $x\in\R^n$. When
speaking of homogeneous polynomials (also called forms), one requires
$f(x)>0$ only for $x\ne(0,\dots,0)$, in order to call $f$ positive
definite.

We shall mostly be working with homogeneous polynomials, except when
it becomes more convenient to dehomogenize. We start with univariate
(inhomogeneous) real polynomials.

\begin{prop}\label{pfprop}%
Let $q\in\R[x]$ be a positive definite polynomial of degree two.
\begin{itemize}
\item[(a)]
Given any psd polynomial $f\in\R[x]$, there are polynomials $\xi$,
$\eta\in\R[x]$ with
\begin{equation}\label{norm}%
f\>=\>\eta^2+q\xi^2.
\end{equation}
\item[(b)]
Assume that $f\ne0$ in (a) satisfies $\deg(f)=2d$. Then the total
number of solutions $(\xi,\eta)$ to \eqref{norm} is $\le2^{d+1}$,
with equality if and only if $q\nmid f$ and $f$ is square-free.
\end{itemize}
\end{prop}

For the proof of Hilbert's theorem we only need part (a). The
second statement will be used in our count of representations.

\begin{proof}
Clearly, $q$ and $f$ may be scaled by any positive real number. By
changing the generator $x$ of the polynomial ring if necessary, we
may therefore assume $q=x^2+1$.

First assume that $f$ is monic of degree~$2$, say $f=(x+a)^2+
b^2$ with real numbers $a$ and $b$. Then $\xi$ as in \eqref{norm} has
to be a constant, and we write $\xi^2=\lambda$. Given $\lambda\in\R$,
the polynomial
$$f-\lambda q\>=\>(1-\lambda)x^2+2ax+(a^2+b^2-\lambda)$$
is a square if and only if either $\lambda=1$, $a=0$ and $b^2\ge1$,
or else $\lambda<1$ and
\begin{equation}\label{3.2}
(1-\lambda)(a^2+b^2-\lambda)-a^2\>=\>0
\end{equation}
(vanishing of the discriminant of $f-\lambda q$). In any case, there
is precisely one value of $\lambda\ge0$ for which $f-\lambda q$ is a
square: For $a=0$, this is $\lambda=\min\{1,b^2\}$, while for $a\ne0$
it is the unique $0\le\lambda<1$ for which \eqref{3.2} vanishes.
(Note that the left hand side of \eqref{3.2} is positive for $\lambda
\gg0$, is $b^2\ge0$ for $\lambda=0$, and is $-a^2<0$ for $\lambda=
1$.) Hence $\xi^2=\lambda$ and $\eta^2=f-q\xi^2$ as in \eqref{3.2}
exist and are unique. Note that there are exactly four possibilities
for the pair $(\xi,\eta)$, except when $f$ or $fq$ is a square. (In
these cases there exist precisely two possibilities, provided
$f\ne0$).

When $f$ is an arbitrary psd polynomial, we can write $f$ as a
product of quadratic psd polynomials. Using the quadratic case just
established, together with the multiplication formul\ae
\begin{equation}\label{multfml}
(a^2+b^2q)(c^2+d^2q)\>=\>(ac\pm bdq)^2+(ad\mp bc)^2q,
\end{equation}
we conclude that $f$ has a representation \eqref{norm}. This proves
(a).

For the proof of (b) we use some basic facts about prime ideal
factorization in Dedekind domains. Let $L=\R(x,\,\sqrt{-q})$, a
quadratic extension of the field $\R(x)$. The integral closure $B$ of
$\R[x]$ in $L$ is a Dedekind domain. It consists of all elements in
$L$ whose norm and trace are in $\R[x]$, from which we see $B=\R[x,\,
\sqrt{-q}]$. The behaviour of the primes in the extension $\R[x]
\subset B$ is easy to see: The linear polynomials $\ell$ in $\R[x]$
are unramified in $B$ and remain prime in $B$, having a quadratic
extension of the residue field. The monic irreducible quadratic
polynomials $p\ne q$ in $\R[x]$ are positive definite, hence they
split into a product $p=p_1p_2$ of two primes in $B$ not associated
to each other, by
\eqref{norm},
while the prime $q$ of $\R[x]$ is ramified. Hence $B$ is a principal
ideal domain. Since $\eta^2+q\xi^2=(\eta+\xi\sqrt{-q})(\eta-\xi
\sqrt{-q})$ is the norm of $\eta+\xi\sqrt{-q}$ in the extension
$\R[x]\subset B$ (for $\xi$, $\eta\in\R[x]$), the number of
representations \eqref{norm} of $f$ is equal to the number of
elements in $B$ of norm $f$.

The norms of the prime elements of $B$ are $N(l)=l^2$, $N(p_1)=N(p_2)
=p$ and $N(\sqrt{-q})=q$. This shows that the number of elements in
$B$ of norm $f$ is obtained as follows: Every factor $p^m$ (for $p\ne
q$ quadratic irreducible) contributes $m+1$ solutions; multiply all
these numbers, and multiply the result by $2$. In other words, the
precise number is (for $f\ne0$)
$$2\prod_p(1+v_p(f)),$$
product over the monic irreducible polynomials $p\ne q$ of degree~$2$.
From this the assertion in (b) is clear.
\end{proof}

It would be possible to present the arguments for part (b) in a way
that avoids using any theory of Dedekind rings. However we felt that
trying this is not worth the effort.

Later it will be preferable for us to use Prop.\ \ref{pfprop} in a
homogenized version. For convenience we state this version here:

\begin{cor}\label{pfproph}%
Let $q\in\R[x,y]$ be a positive definite quadratic form. Given any
psd form $f\in\R[x,y]$ of degree $2d$, there exist forms $\xi$, $\eta
\in\R[x,y]$ with $\deg(\xi)=d-1$, $\deg(\eta)=d$ and $f=\eta^2+
q\xi^2$. The number of such pairs $(\xi,\eta)$ is $\le2^{d+1}$, with
equality if and only if $q\nmid f$ and $f$ is square-free.
\qed
\end{cor}


\section{The case where $f$ has a real zero}\label{elemreal0}%

\begin{lab}
Let $f=f(x,y,z)$ be a psd quartic form in $\R[x,y,z]$, and assume
that $f=0$ has a nontrivial real zero. Changing coordinates linearly
we can assume $f(0,0,1)=0$, hence
\begin{equation}\label{nfreal0}
f\>=\>f_2(x,y)\cdot z^2+f_3(x,y)\cdot z+f_4(x,y)
\end{equation}
where $f_j=f_j(x,y)$ is a binary form of degree $j$ ($j=2,3,4$). That
$f$ is psd means that each of the three binary forms
$$f_2,\ f_4,\ 4f_2f_4-f_3^2$$
is psd, that is, a sum of two squares. By an argument which is
entirely elementary and explicit, we shall construct a representation
of $f$ as a sum of three squares (Proposition \ref{exrep}). For
generically chosen $f_2$, $f_3$, $f_4$, we shall in fact construct
all such representations (Proposition \ref{summreal0}). This second
part is not needed for the proof of Hilbert's theorem.
\end{lab}

\begin{lab}\label{pfkonstr1}%
Let us start by showing that $f$ is a sum of three squares.
If $f_2=0$ then also $f_3=0$, and hence $f=f_4$ is a psd binary form,
therefore a sum of two squares. If $0\ne f_2=l^2$ is a square of a
linear form, then $4l^2f_4\ge f_3^2$ shows $l\mid f_3$, say $f_3=
2lg_2$. Observe that $f_4-g_2^2$ is a sum of two squares since $4l^2
(f_4-g_2^2)=4f_2f_4-f_3^2$ is a sum of two squares. Therefore $f=
(lz+g_2)^2+(f_4-g_2^2)$ is a sum of three squares.
\end{lab}

\begin{lab}\label{pfkonstr2}%
It remains to discuss the case where $f_2$ is strictly positive
definite. From Cor.\ \ref{pfproph} we see that there exist binary
forms $\xi=\xi(x,y)$ and $\eta=\eta(x,y)$ with $\deg(\xi)=2$, $\deg
(\eta)=3$ and $\eta^2+\xi^2f_2=4f_2f_4-f_3^2$, that is,
\begin{equation}\label{pfbeding}%
\eta^2+f_3^2=f_2(4f_4-\xi^2).
\end{equation}
On the other hand, since $f_2$ is psd, there are linear forms $l_1$,
$l_2\in\R[x,y]$ with $f_2=l_1^2+l_2^2=(l_1+il_2)(l_1-il_2)$ ($i^2=
-1$). By similarly factoring the left hand side of \eqref{pfbeding},
it follows that $l_1+il_2$ divides one of $\eta\pm if_3$. Replacing
$l_2$ by $-l_2$ if necessary we can assume
$$(l_1+il_2)\>\mid\>(\eta+if_3).$$
This implies that $f_2$ divides $(\eta+if_3)(l_1-il_2)=(\eta l_1+
f_3l_2)+i(f_3l_1-\eta l_2)$. Hence $f_2$ divides both real and
imaginary part of the right hand form. So the fractions
$$h_1:=\frac{f_3l_1-\eta l_2}{2f_2},\qquad h_2:=\frac{\eta l_1+
f_3l_2}{2f_2}$$
are binary quadratic forms (with real coefficients), and
\eqref{pfbeding} implies
$$h_1^2+h_2^2=\frac{(\eta^2+f_3^2)(l_1^2+l_2^2)}{4f_2^2}=\frac
{\eta^2+f_3^2}{4f_2}=f_4-\frac14\xi^2.$$
Moreover
$$h_1l_1+h_2l_2\>=\>\frac{f_3(l_1^2+l_2^2)}{2f_2}\>=\>\frac12\,f_3,$$
and so
$$f=\Bigl(\frac\xi2\Bigr)^2+(h_1+l_1z)^2+(h_2+l_2z)^2$$
is a sum of three squares of quadratic forms. We have thus proved:
\end{lab}

\begin{prop}\label{exrep}
Let $f\in\R[x,y,z]$ be a psd quartic form which has a nontrivial real
zero. Then $f$ is a sum of three squares of quadratic forms in $\R[x,
y,z]$.
\qed
\end{prop}

Note that the proof was entirely explicit and constructive.

We now turn to the task of determining all representations of $f$, at
least in the case when $f_2$, $f_3$, $f_4$ are chosen generically.
For this, the following definition is useful.

\begin{dfn}\label{dfnortheq}
Two representations
$$f\>=\>\sum_{i=1}^3p_i^2\>=\>\sum_{i=1}^3p_i'^2$$
with quadratic forms $p_i$, $p'_i\in\R[x,y,z]$ are said to be
\emph{(orthogonally) equivalent} if there exists an orthogonal matrix
$S=(s_{ij})\in\O_3(\R)$ such that
$$p_j'\>=\>\sum_{i=1}^3s_{ij}p_i\quad(j=1,2,3).$$
\end{dfn}

\begin{lab}\label{rep2xieta}
Let $f=f_2z^2+f_3z+f_4$ be a psd form as in \eqref{nfreal0}. We
assume that $f_2$ is not a square, hence is strictly positive
definite. Assume
\begin{equation}\label{darst}
f=\sum_{i=1}^3(v_iz+w_i)^2
\end{equation}
where $v_i$ resp.\ $w_i\in\R[x,y]$ are homogeneous of respective
degrees $1$ resp.~$2$ ($i=1,2,3$). We first show how to associate
with \eqref{darst} a solution $(\xi,\eta)$ of \eqref{pfbeding}.

Consider the column vectors $v=(v_1,v_2,v_3)^t$ and $w=(w_1,w_2,w_3)
^t$ with polynomial entries. Since the linear forms $v_1$, $v_2$,
$v_3$ are linearly dependent, there is an orthogonal matrix $S\in\O_3
(\R)$ such that the first entry of the column $Sv$ is zero. Replacing
$v$ resp.\ $w$ by $Sv$ resp.\ $Sw$ yields an equivalent
representation $f=\sum_{i=1}^3(v'_iz+w'_i)^2$ in which $v'_1=0$. So
up to replacing \eqref{darst} by an equivalent representation we can
assume $v_1=0$, and get accordingly
$$f_2=v_2^2+v_3^2,\quad f_3=2(v_2w_2+v_3w_3),\quad f_4=w_1^2+w_2^2+
w_3^2.$$
Putting $\xi:=2w_1$ and $\eta:=2(v_2w_3-v_3w_2)$ gives
\begin{align*}
\eta^2+f_3^2 & = 4(v_2w_3-v_3w_2)^2+4(v_2w_2+v_3w_3)^2 \\
& = 4(v_2^2+v_3^2)(w_2^2+w_3^2) \\
& = f_2(4f_4-\xi^2)
\end{align*}
so $(\xi,\eta)$ solves \eqref{pfbeding}.

Note that a different choice of $S$ does not change $\xi^2$ and
$\eta^2$. Indeed, the first row of $S$ is unique up to a factor
$\pm1$ since $v_1$, $v_2$, $v_3$ span the space of linear forms in
$\R[x,y]$. Therefore $\pm\xi$ does not change if $S$ is chosen
differently. The same argument shows that $\xi^2$ and $\eta^2$ depend
only on the equivalence class of \eqref{darst}.
\end{lab}

\begin{lab}\label{gener164}
When $f_2$ is not a square, note that the number of solutions $(\xi,
\eta)$ of \eqref{pfbeding} was determined in Prop.\ \ref{pfprop}(b).
In particular, it was shown there that this number is $\le16$, and is
equal to $16$ if and only if $f_2\nmid f_3$ and $4f_2f_4-f_3^2$ is
square-free. In this latter case, the pair $(\xi^2,\eta^2)$ can
therefore take precisely four different values.
\end{lab}

\begin{lab}\label{eindtfall0}
Assume that $f_2$ is not a square, that $f_2\nmid f_3$ and $4f_2f_4
-f_3^2$ is square-free. We show that inequivalent representations
\eqref{darst} give different solutions $(\xi^2,\eta^2)$ to
\eqref{pfbeding}. Combined with \ref{gener164}, this will imply that
$f$ has precisely four different representations up to equivalence.

Let
\begin{align*}
f & = w_1^2+(v_2z+w_2)^2+(v_3z+w_3)^2 \\
& =  w_1'^2+(v_2'z+w_2')^2+(v_3'z+w_3')^2
\end{align*}
be two representations with the same invariants $\xi^2$, that is,
with $w_1^2=w_1'^2=\frac{\xi^2}4$. Then $v_2w_3-v_3w_2=\pm(v'_2w'_3
-v'_3w'_2)$, and we can assume
$$v_2w_3-v_3w_2=v'_2w'_3-v'_3w'_2$$
by multiplying $v_2z+w_2$ with $-1$ if necessary. Writing $v=v_2+
iv_3$, $w=w_2+iw_3$ and $v'=v'_2+iv'_3$, $w'=w'_2+iw'_3$ this means
$\Im(\ol vw)=\Im(\ol v'w')$. On the other hand we have
$$v\ol v=v'\ol v'=f_2,\quad\Re(\ol vw)=\Re(\ol v'w')=\frac12f_3,\quad
4w\ol w=4w'\ol w'=4f_4-\xi^2,$$
and we conclude
\begin{equation}\label{vws}
\ol vw=\ol v'w'.
\end{equation}
Now $\ol v$ does not divide $w'$, because otherwise $v\ol v=f_2$
would divide $4w'\ol w'=4f_4-\xi^2$, and hence we would have
$$f_2^2\mid(\eta^2+f_3^2)=(\eta+if_3)(\eta-if_3),$$
whence $f_2\mid f_3$, which was excluded. Comparing the two products
\eqref{vws} we see that there exist $\lambda$, $\mu\in\C$ with $v'=
\lambda v$ and $w'=\mu w$, and clearly we must have $|\lambda|=|\mu|
=1$. Therefore \eqref{vws} shows $\lambda=\mu$. This means that the
two representations we started with are equivalent.
\end{lab}

We summarize these discussions:

\begin{prop}\label{summreal0}
Let $f=f_2z^2+f_3z+f_4$ be psd (with $f_i\in\R[x,y]$ homogeneous of
degree~$i$, for $i=2,3,4$), and assume that $f_2$ is not a square.
\begin{itemize}
\item[(a)]
Associated with each representation of $f$ as a sum of three squares
is a well-defined solution of
$$\eta^2+f_3^2\>=\>f_2(4f_4-\xi^2)$$
such that $\xi^2$ and $\eta^2$ depend only on the orthogonal
equivalence class of the representation.
\item[(b)]
If $f_2\nmid f_3$ and $4f_2f_4-f_3^2$ is square-free, then any two
representations of $f$ with the same invariants $\xi^2$, $\eta^2$ are
equivalent. There exist precisely four different equivalence classes
of representations of $f$.
\qed
\end{itemize}
\end{prop}

\begin{rem}
Let $f=f_2z^2+f_3z+f_4$ be psd, as in Proposition \ref{summreal0}.
The real zero $(0,0,1)$ is a singularity of the projective curve
$f=0$. That $f_2$ is not a square means that this singularity is a
node (with two complex conjugate tangents). When $f_2\nmid f_3$ and
$4f_2f_4-f_3^2$ is square-free, one can show that $(0,0,1)$ is the
only singularity of the curve (the converse is not true). The fact
that $f$ has precisely four inequivalent representations is in
agreement with the results of \cite{sch:jag}.
\end{rem}


\section{The case where $f$ has no real zero}\label{sec:noreal0}

The following normalization lemma was proved in \cite{Pf}:

\begin{lem}\label{normalisform}%
Let $f=f(x,y,z)$ be a strictly positive definite form of degree four
in $\R[x,y,z]$. Then, by a linear change of coordinates, $f$ can be
brought into the form
\begin{equation}\label{normalisf}%
f\>=\>z^4+f_2z^2+f_3z+f_4
\end{equation}
in which $f_j\in\R[x,y]$ is a form of degree $j$ ($j=2,3,4$), and
such that the form $f-z^4$ is psd.
\end{lem}

\begin{proof}
Let $c>0$ be the minimum value taken by $f$ on the unit sphere $S^2$
in $\R^3$. Scaling $f$ with a positive factor we may
assume $c=1$, and after an orthogonal coordinate change we get $c=1=
f(0,0,1)$. The form $\tilde f:=f-(x^2+y^2+z^2)^2$ is nonnegative on
$\R^3$ and vanishes at $(0,0,1)$. Therefore $\tilde f$ does not
contain the term $z^4$, in fact $\deg_z(\tilde f)\le2$. This means
that $f$ has the shape \eqref{normalisf}. The last assertion follows
from $f-z^4=\tilde f+(x^2+y^2+z^2)^2-z^4\ge\tilde f\ge0$.
\end{proof}

\begin{rems}\label{findmin}%
1.\
The form $f-z^4$ is psd and vanishes in $(0,0,1)$, so the results of
Sect.~\ref{elemreal0} apply to $f-z^4$. In particular, we can
explicitly construct a representation of $f-z^4$ as a sum of three
squares.
\smallskip

2.\
The minimum value of $f$ on the unit sphere can be found by
inspecting the solutions of
the equation $\nabla f(x,y,z)=\lambda\cdot(x,y,z)$ with $\lambda\in
\R$.
\end{rems}

For $f$ as in \eqref{normalisf} we now study the question when $f$ is
a sum of three squares.

\begin{prop}[\cite{Pf} Prop.\ 3.1]\label{pfthm}%
Let $f=z^4+f_2z^2+f_3z+f_4$ where $f_j\in\R[x,y]$ is a form of
degree~$j$ ($j=2,\,3,\,4$). Then $f$ is a sum of three squares if,
and only if, there exist binary forms $\xi$, $\eta\in\R[x,y]$ with
$\deg(\xi)=2$, $\deg(\eta)=3$ and
\begin{equation}\label{xietaglng}%
\eta^2+f_3^2\>=\>(f_2-\xi)(4f_4-\xi^2),
\end{equation}
such that
\begin{equation}\label{xietaxtra}%
f_2-\xi\ge0,\quad4f_4-\xi^2\ge0.
\end{equation}
\end{prop}

\begin{rem}\label{rempsdconds}
If one of $f_2-\xi$ and $4f_4-\xi^2$ is psd, then so is the other by
\eqref{xietaglng}, except possibly in the case where $f_2-\xi$ resp.\
$4f_4-\xi^2$ was zero. The latter can happen only if $f_3=0$ and
$\eta=0$. Note that the psd conditions in \eqref{xietaxtra} mean that
the two forms are sums of two squares of linear resp.\ quadratic
forms.
\end{rem}

\begin{proof}[Proof of \ref{pfthm}]
First assume $f=\sum_{i=1}^3(u_iz^2+v_iz+w_i)^2$, where $u_i$, $v_i$,
$w_i\in\R[x,y]$ are forms of respective degrees $0$, $1$,~$2$ ($1\le
i\le3$). The vector $(u_1,u_2,u_3)\in\R^3$ has unit length, so by
changing with an orthogonal real $3\times3$~matrix we can get $u_1=1$
and $u_2=u_3=0$. This implies $v_1=0$, $v_2^2+v_3^2=f_2-2w_1$,
$2(v_2w_2+v_3w_3)=f_3$ and $w_2^2+w_3^2=f_4-w_1^2$. One checks that
\eqref{xietaglng} and \eqref{xietaxtra} are satisfied with
$$\xi=2w_1,\quad\eta=2(v_2w_3-v_3w_2).$$
Conversely assume that $\xi$, $\eta$ satisfy \eqref{xietaglng} and
\eqref{xietaxtra}. If $\xi=f_2$, then $f_3=0$, and by
\eqref{xietaxtra} there are quadratic forms $w_2$, $w_3\in\R[x,y]$
with $f_4-\frac14f_2^2=w_2^2+w_3^2$, so
$$f=z^4+f_2z^2+f_4=\Bigl(z^2+\frac{f_2}2\Bigr)^2+w_2^2+w_3^2.$$
Now assume $\xi\ne f_2$. By \eqref{xietaxtra} there are linear forms
$v_2$, $v_3\in\R[x,y]$ with $f_2-\xi=v_2^2+v_3^2=(v_2+iv_3)
(v_2-iv_3)$ (where $i^2=-1$). From \eqref{xietaglng} we see that the
linear form $v_2+iv_3$ divides one of the two forms $\eta\pm if_3$
(in $\C[x,y]$). Replacing $v_3$ with $-v_3$ if necessary we can
assume $(v_2+iv_3)\mid(\eta+if_3)$. This implies that $f_2-\xi$
divides
$$(\eta+if_3)(v_2-iv_3)\>=\>(f_3v_3+\eta v_2)+i(f_3v_2-\eta v_3).$$
Therefore,
$$\Bigl(z^2+\frac\xi2\Bigr)^2+\Bigl(v_2z+\frac{f_3v_2-\eta v_3}
{2(f_2-\xi)}\Bigr)^2+\Bigl(v_3z+\frac{f_3v_3+\eta v_2}{2(f_2-\xi)}
\Bigr)^2$$
is a sum of three squares in $\R[x,y,z]$. A comparison of the
coefficients shows that this sum is equal to~$f$.
\end{proof}

\begin{rem}
Consider $f=z^4+f_2z^2+f_3z+f_4$ as a monic polynomial in $z$, with
coefficients $f_j\in\R[x,y]$ as in Prop.\ \ref{pfthm}. Equation
\eqref{xietaglng} says $\eta^2=r_f(\xi)$ where
$$r_f(z)\>=\>(f_2-z)(4f_4-z^2)-f_3^2$$
is the cubic resolvent of $f$ with respect to $z$ (see
\ref{discrfquart} below).
\end{rem}

The following lemma follows from the fact that the sum of squares map
$(p_1,p_2,p_3)\mapsto\sum_jp_j^2$ is topologically proper (see (1) of
the introduction). Avoiding this argument we give a direct proof
based on Prop.\ \ref{pfthm}:

\begin{lem}\label{limarg}
Let $f^{(1)},\>f^{(2)},\>\dots$ be a sequence of quartic forms as in
\ref{pfthm} which converges coefficient-wise to a form $f$. If every
$f^{(j)}$ is a sum of three squares, then the same is true for $f$.
\end{lem}

\begin{proof}
For every index $j$ there exist forms $\xi^{(j)}$, $\eta^{(j)}\in
\R[x,y]$ satisfying the conditions of Prop.\ \ref{pfthm}. From the
inequality $\bigl(\xi^{(j)}\bigr)^2\le4f_4$ it follows that the
sequence $\xi^{(j)}$ is bounded, and so the sequence $\eta^{(j)}$ is
bounded as well. Hence there exists a limit point $(\xi,\eta)$ of the
sequence $\bigl(\xi^{(j)},\>\eta^{(j)}\bigr)$, and $(\xi,\eta)$
satisfies the conditions of \ref{pfthm} for the form $f$.
\end{proof}

The rest of this section is not needed for our proof of Hilbert's
theorem. Similar as in the case where $f$ has a real zero (Section
\ref{elemreal0}), we try to find all representations of $f$ as a sum
of three squares.

\begin{lem}\label{xietawelldf}
Let $f$ be as in Prop.\ \ref{pfthm}. The construction in the proof of
Prop.\ \ref{pfthm} associates with every representation
\begin{equation}\label{fixtrep}
f\>=\>p_1^2+p_2^2+p_3^2
\end{equation}
a pair $(\xi,\eta)$ which solves \eqref{xietaglng} and
\eqref{xietaxtra}. The form $\xi$ is independent from the choices. In
fact it depends only on the orthogonal equivalence class of the
representation \eqref{fixtrep}.
\end{lem}

\begin{proof}
Consider the representation \eqref{fixtrep}, and write
$$p_i=u_iz^2+v_iz+w_i\quad(i=1,2,3)$$
where $u_i$, $v_i$, $w_i\in\R[x,y]$ are homogeneous of respective
degrees $0$, $1$ and~$2$. Writing $u=(u_1,u_2,u_3)^t$, $v=(v_1,v_2,
v_3)^t$, $w=(w_1,w_2,w_3)^t$, we choose $S\in\O_3(\R)$ with $Su=
(1,0,0)^t$ as in the proof of Prop.\ \ref{pfthm}. If $Sv=(v'_1,v'_2,
v'_3)^t$ and $Sw=(w'_1,w'_2,w'_3)^t$, we have shown that
$$\xi=2w'_1,\quad\eta=2(v'_2w'_3-v'_3w'_2)$$
solve \eqref{xietaglng} and \eqref{xietaxtra}. If $T$ is another
orthogonal matrix with $Tu=(1,0,0)^t$, then $T=US$ where $U$ is
orthogonal with first column and row $(1,0,0)$. This shows that using
$T$ instead of $S$ does not change $\xi$.
The same argument shows that $\xi$ depends only on the orthogonal
equivalence class of \eqref{fixtrep}.
\end{proof}

\begin{prop}\label{eindtfall1}
Let $f=z^4+f_2z^2+f_3z+f_4$ with $f_j\in\R[x,y]$ homogeneous of
degree~$j$ ($j=2,3,4$), and assume $\gcd(f_3,\,4f_4-f_2^2)=1$. Let
$$f\>=\>\sum_{i=1}^3p_i^2\>=\>\sum_{i=1}^3p_i'^2$$
be two representations of $f$ with associated invariants $\xi$ and
$\xi'$ (see Lemma \ref{xietawelldf}). If $\xi=\xi'$,
the two representations are orthogonally equivalent.
\end{prop}

\begin{proof}
Assuming $f_2-\xi\ne0$, we first show that $f_2-\xi$ does not divide
$4f_4-\xi^2$. From
$$(f_2-\xi)(4f_4-\xi^2)=\eta^2+f_3^2=(\eta+if_3)(\eta-if_3)$$
we see that $(f_2-\xi)\mid(4f_4-\xi^2)$ would imply $(f_2-\xi)\mid
f_3$.
On the other hand, it would imply $(f_2-\xi)\mid(4f_4-f_2^2)$, thus
contradicting the assumption.

Write $p_i=u_iz^2+v_iz+w_i$ and $p'_i=u'_iz^2+v'_iz+w'_i$ ($i=1,2,3$)
as in the proof of Lemma \ref{xietawelldf}.
We can assume $u_1=u'_1=1$ and $u_i=u'_i=0$ for $i=2,3$. By
hypothesis we have $w_1=w'_1=\frac\xi2$ and $v_2w_3-v_3w_2=\pm
(v'_2w'_3-v'_3w'_2)$; replacing $p_3$ with $-p_3$ if necessary we can
assume
\begin{equation}\label{etas}
v_2w_3-v_3w_2\>=\>v'_2w'_3-v'_3w'_2\>=\>\frac\eta2\,.
\end{equation}
Since the coefficient of $z^3$ vanishes in $f$ we have $v_1=v'_1=0$,
and so $p_1=p'_1=z^2+\frac\xi2$. Write $v:=v_2+iv_3$, $w:=w_2+iw_3$
and similarly $v':=v'_2+iv'_3$, $w':=w'_2+iw'_3$. Then \eqref{etas}
says $\Im(\ol vw)=\Im(\ol v'w')=\frac12\eta$. A comparison of the
other coefficients gives $v\ol v=v'\ol v'=f_2-\xi$, $\Re(\ol vw)=
\Re(\ol v'w')=\frac12f_3$ and $4w\ol w=4w'\ol w'=4f_4-\xi^2$.
In particular, $\ol vw=\ol v'w'=\frac12(f_3+i\eta)$.

We clearly have $v=0$ $\iff$ $v'=0$, and similarly $w=0$ $\iff$ $w'=
0$. In either of these cases, it is clear that the two
representations are equivalent. Hence we can assume $v$, $w\ne0$. Now
$\ol v$ does not divide $w'$, because otherwise $v\ol v\mid w'
\ol w'$, i.e., $(f_2-\xi)\mid(4f_4-\xi^2)$, which was ruled out at
the beginning.
So we conclude that there exist $\lambda$, $\mu\in\C$ with $|\lambda|
=|\mu|=1$ and $v'=\lambda v$, $w'=\mu w$. Then $\ol vw=\ol v'w'$
implies $\lambda=\mu$. Hence the two representations are orthogonally
equivalent.
\end{proof}

\begin{cor}\label{eindtfall2}
If $\gcd(f_3,\,4f_4-f_2^2)=1$ the number of inequivalent
representations of $f$ equals the number of forms $\xi$ solving
\eqref{xietaglng} and \eqref{xietaxtra} with suitable $\eta$.
\qed
\end{cor}


\section{Deforming the quartic, I}\label{sec:defqui}

\begin{lab}\label{normlf}%
Let $f=f(x,y,z)$ be a nonzero psd quartic form with real
coefficients. We are trying to show that $f$ is a sum of three
squares. The case where $f$ has a nontrivial real zero has already
been solved completely. From now on we assume that $f$ is strictly
positive definite. We shall use a deformation to a suitable psd form
with real zero to arrive at the desired conclusion, at least in a
generic case.

As in Lemma \ref{normalisform} we use
scaling by a positive number and an orthogonal coordinate change to
bring $f$ into the form
\begin{equation}\label{normalform}%
f\>=\>z^4+f_2(x,y)z^2+f_3(x,y)z+f_4(x,y)
\end{equation}
with $\deg(f_j)=j$ ($j=2,3,4$), such that the form
$$f-z^4\>=\>f_2(x,y)z^2+f_3(x,y)z+f_4(x,y)$$
is psd. The latter means that each of the binary forms $f_2$, $f_4$
and $4f_2f_4-f_3^2$ is psd.
\end{lab}

\begin{lab}
Let $t$ be a real parameter. Fixing $f$ as in \ref{normlf}, we
consider the family of quartic forms
\begin{align}
f^{(t)} & := t^2f+(1-t^2)(f-z^4) \notag \\
& \phantom:= t^2z^4+f_2(x,y)z^2+f_3(x,y)z+f_4(x,y) 
\end{align}\label{ft}%
($t\in\R$). For $0<|t|\le1$, the form $f^{(t)}$ is strictly
positive definite, while $f^{(0)}=f-z^4$ has a zero at $(0,0,1)$.
When $t$ runs from $0$ to $1$, the form $f^{(t)}$ covers the line
segment between $f-z^4$ and $f$ (inside the space of all real
quartic forms). Note however that the time parameter is quadratic,
not linear.
\end{lab}

\begin{lab}\label{ftso3s}%
Let $0\ne t\in\R$. By Prop.\ \ref{pfthm}, $f^{(t)}$ is a sum of three
squares if and only if there are forms $\tilde\xi$, $\tilde\eta\in
\R[x,y]$ with $\tilde\eta^2+t^{-4}f_3^2=(t^{-2}f_2-\tilde\xi)(4t^{-2}
f_4-\tilde\xi^2)$, with both factors on the right psd. Multiplying
with $t^4$ and substituting $\xi=t\tilde\xi$, $\eta=t^2\tilde\eta$,
we see that this happens if and only if there are forms $\xi$, $\eta$
in $\R[x,y]$ (of degrees $2$ resp.~$3$) such that
\begin{equation}\label{xietaglngt1}%
\eta^2+f_3^2\>=\>(f_2-t\xi)(4f_4-\xi^2),
\end{equation}
\begin{equation}\label{xietaglngt2}%
f_2-t\xi\ge0,\ \ 4f_4 -\xi^2\ge0.
\end{equation}
On the other hand, conditions \eqref{xietaglngt1}, \eqref{xietaglngt2}
have a solution $(\xi_0,\eta_0)$ for $t=0$,
provided that $f_2$ is not a square, since $4f_2f_4-f_3^2$ is then
represented by $\lqf1,f_2\rqf$ (Cor.\ \ref{pfproph}, see also
\eqref{pfbeding} above). The condition $4f_4-\xi_0^2\ge0$ is
automatic since $f_2\ge0$ and $f_2\ne0$. Keeping the assumption that
$f_2$ is not a square, let us fix such forms $\xi_0$, $\eta_0$ with
$\deg(\xi_0)=2$, $\deg(\eta_0)=3$ and
\begin{equation}\label{xietat0}%
\eta_0^2+f_2\xi_0^2\>=\>4f_2f_4-f_3^2.
\end{equation}
\end{lab}

\begin{prop}\label{fortsezt0}%
In addition to the assumptions in \ref{normlf}, assume that $f_2$ is
not a square, that $f_2\nmid f_3$ and that $4f_2f_4-f_3^2$ is
square-free. Then there exist continuous families $(\xi^{(t)})$,
$(\eta^{(t)})$ ($|t|<\epsilon$, for some $\epsilon>0$) of forms such
that $(\xi^{(0)},\eta^{(0)})=(\xi_0,\eta_0)$, and such that $(\xi^
{(t)},\eta^{(t)})$ solves \eqref{xietaglngt1}, \eqref{xietaglngt2}
for all $|t|<\epsilon$.
\end{prop}

For the proof we need the following simple lemma:

\begin{lem}\label{relprimlem}%
Let $k$ be a field, let $f$, $g\in k[t]$ be polynomials with $\deg(f)
=m$, $\deg(g)=n$ and $m,\,n\ge1$. The linear map
$$k[t]_{m-1}\oplus k[t]_{n-1}\to k[t]_{m+n-1},\quad(p,q)\mapsto
pg+qf$$
is bijective if and only if $f$ and $g$ are relatively prime. (Here
$k[t]_d$ denotes the space of polynomials of degree $\le d$.)
\end{lem}

\begin{proof}
Both the source and the target vector space have the same dimension
$m+n$. If $f$ and $g$ are relatively prime, then $pg+qf=0$ implies
$f\mid p$ and $g\mid q$, whence $p=q=0$ by degree reasons.
The reverse implication is obvious.
\end{proof}

Note that if one uses the canonical linear bases to describe the map
in \ref{relprimlem} by a matrix, and takes its determinant, one
obtains the resultant of $f$ and~$g$.

\begin{proof}[Proof of Prop.\ \ref{fortsezt0}]
We first exploit the assumption. The forms $\xi_0$ and $\eta_0$ are
relatively prime since the square of any common divisor divides
$4f_2f_4-f_3^2$ by \eqref{xietat0}. Also, the irreducible form $f_2$
does not divide $\eta_0$, since otherwise \eqref{xietat0} would imply
$f_2\mid f_3$. We conclude that $f_2\xi_0$ and $\eta_0$ are relatively
prime.

Let $V_d\subset\R[x,y]$ denote the space of binary forms of
degree~$d$, and consider the map
$$F\colon V_2\times V_3\times\R\to V_6,\quad(\xi,\eta,t)\>\mapsto\>
\eta^2+f_3^2-(f_2-t\xi)(4f_4-\xi^2).$$
The partial derivative of $F$ at $(\xi_0,\eta_0,0)$ with respect to
$(\xi,\eta)$ is the linear map
$$V_2\oplus V_3\to V_6,\quad(\xi,\eta)\mapsto2(\eta_0\eta-f_2\xi_0
\xi).$$
By Lemma \ref{relprimlem}, this map is bijective.

The theorem on implicit functions gives us therefore the existence of
continuous families $(\xi^{(t)})$, $(\eta^{(t)})$, for $|t|<
\epsilon'$ and some $\epsilon'>0$, with $(\xi^{(0)},\eta^{(0)})=
(\xi_0,\eta_0)$ and with $F(\xi^{(t)},\eta^{(t)},\,t)=0$ (that is,
\eqref{xietaglngt1}) for $|t|<\epsilon'$. As for conditions
\eqref{xietaglngt2}, it suffices to verify the first of them since
$f_3\ne0$. For $t=0$, $f_2-t\xi^{(t)}=f_2$ is strictly positive
definite by assumption. Hence there is some $\epsilon''>0$ such that
$f_2-t\xi^{(t)}\ge0$ for all $|t|<\epsilon''$, and we can take
$\epsilon=\min\{\epsilon',\epsilon''\}$.
\end{proof}

Using Prop.\ \ref{pfthm}, we conclude from Prop.\ \ref{fortsezt0}:

\begin{cor}\label{deform1}%
Assume that $f=z^4+f_2z^2+f_3z+f_4$ (with $f_j\in\R[x,y]$ and $\deg
(f_j)=j$) is strictly positive definite and satisfies $f-z^4\ge0$. If
$f_2$ is not a square, $f_2\nmid f_3$ and $4f_2f_4-f_3^2$ is
square-free, then there exists $\epsilon>0$ such that $f^{(t)}$ is a
sum of three squares for all $0\le|t|<\epsilon$.
\end{cor}

\begin{rems}\label{contvsanal}
1.\
It can be shown that a representation of $f^{(t)}$ as a sum of three
squares can be chosen for every $|t|<\epsilon$ such that the
polynomials in this representation depend continuously on $t$,
starting at $t=0$ with an arbitrary representation of $f^{(0)}=
f-z^4$.
\smallskip

2.\
Since the map $F$ in the proof of \ref{fortsezt0} is polynomial, a
suitable version of the implicit function theorem (see \cite{D}
10.2.4, for example)
shows that the families $(\xi^{(t)})$, $(\eta^{(t)})$ are not just
continuous but even analytic.
\end{rems}


\section{The discriminant}\label{sec:disc}

Before proceeding to extend representations of $f^{(t)}$ over the
entire interval $0\le t\le1$, we need to discuss the discriminant of
$f^{(t)}$.

\begin{lab}\label{discrappel}
Here are some reminders about the classical discriminant. Let $K$ be
a field, let
$$f\>=\>a_0z^n+a_1z^{n-1}+\cdots+a_n\ \in K[z]$$
with $a_0\ne0$. The discriminant of $f$ is defined as
$$\disc(f)\>=\>\disc_n(f)\>=\>a_0^{2n-2}\prod_{i<j}(\alpha_i
-\alpha_j)^2$$
if $\alpha_1,\dots,\alpha_n$ are the roots of $f$ in an algebraic
closure of $K$. More precisely, this is the $n$-discriminant of $f$;
if $\deg(f)=m<n-1$ then $\disc_n(f)=0$, while in general $\disc_m(f)
\ne0$. If $\deg(f)=n$ then it follows directly from the definition
that $\disc_n(f)=0$ if and only if $f$ has a multiple root.

Using the theorem on symmetric functions one sees that $\disc_n(f)$
is an integral polynomial in the coefficients $a_0,\dots,a_n$ of $f$.
Moreover, there exist universal polynomials $p$, $q\in\Z[a_0,\dots,
a_n,z]$ such that
$$\disc_n(f)\>=\>pf+qf'$$
where $f'$ is the derivative of $f$. One finds $p$ and $q$ by writing
$f$ with indeterminate coefficients and performing the Euclidean
algorithm on $f$ and $f'$.

Directly from the definition one sees that the polynomial $f(\lambda
z)$ has discriminant
\begin{equation}\label{discrvarsubst}
\disc_nf(\lambda z)\>=\>\lambda^{n(n-1)}\,\disc_nf(z),
\end{equation}
if $\lambda\in K$ is a parameter.
\end{lab}

In the remainder of this section the degree $n$ is always clear from
the context, and we omit the index $n$ of the discriminant.

\begin{lab}\label{discrfquart}
Given a quartic polynomial
$$f(z)\>=\>a_0z^4+a_1z^3+a_2z^2+a_3z+a_4,$$
the cubic resolvent $r_f(z)$ of $f(z)$ is defined to be the cubic
polynomial
$$r_f(z)\>=\>a_0^3z^3-a_0^2a_2z^2+a_0(a_1a_3-4a_0a_4)z+(4a_0a_2a_4
-a_0a_3^2-a_1^2a_4).$$
If $a_0\ne0$ and $\alpha_1,\dots,\alpha_4$ are the roots of $f(z)$ in
an algebraic closure of $K$, a calculation with symmetric polynomials
shows that $r_f(z)$ has the roots
$$\beta_1=\alpha_1\alpha_2+\alpha_3\alpha_4,\quad\beta_2=\alpha_1
\alpha_3+\alpha_2\alpha_4,\quad\beta_3=\alpha_1\alpha_4+\alpha_2
\alpha_3.$$
We will only use the case where the cubic coefficient $a_1$ of $f$
vanishes, in which
$$\disc(f)\>=\>a_0\bigl(-4a_2^3a_3^2-27a_0a_3^4+16a_2^4a_4-128a_0
a_2^2a_4^2+144a_0a_2a_3^2a_4+256a_0^2a_4^3\bigr)$$
and
$$r_f(z)\>=\>a_0\Bigl((a_0z-a_2)(a_0z^2-4a_4)-a_3^2\Bigr).$$
\end{lab}

\begin{lem}\label{discrfr}
Let $f=a_0z^4+a_1z^3+a_2z^2+a_3z+a_4$. Then
$$\disc r_f(z)\>=\>a_0^6\,\disc f(z).$$
\end{lem}

\begin{proof}
If $\beta_1$, $\beta_2$, $\beta_3$
are the roots of $r_f$ as in \ref{discrfquart}, then
\begin{align*}
\beta_1-\beta_2=(\alpha_1-\alpha_4)(\alpha_2-\alpha_3), \\
\beta_1-\beta_3=(\alpha_1-\alpha_3)(\alpha_2-\alpha_4), \\
\beta_2-\beta_3=(\alpha_1-\alpha_2)(\alpha_3-\alpha_4),
\end{align*}
from which one immediately sees
$$\disc(r_f)=a_0^{12}\prod_{1\le k<l\le3}(\beta_k-\beta_l)^2=a_0^{12}
\prod_{1\le i<j\le4}(\alpha_i-\alpha_j)^2=a_0^6\disc(f).$$
\end{proof}

\begin{rem}\label{disc0posdef}%
If $A=\R$, and if a quartic polynomial $f(z)=a_0z^4+a_2z^2+a_3z+a_4
\in\R[z]$ with $a_0\ne0$ is known to be strictly positive definite,
then $f$ can have a multiple root only if $f$ is a square. Therefore,
$\disc(f)=0$ is equivalent to $a_3=a_2^2-4a_0a_4=0$ in this case.
\end{rem}

\begin{lab}\label{dfngtht}
Now let $t\ne0$ be a real parameter. We consider
$$f^{(t)}\>=\>t^2z^4+f_2(x,y)z^2+f_3(x,y)z+f_4(x,y)$$
as a quartic polynomial in the variable $z$ over $\R[x,y]$ (see
\eqref{ft}). Let $r^{(t)}$ be the cubic resolvent of $f^{(t)}$ (with
respect to $z$). We put
$$g_t(z)\>:=\>\frac1{t^2}\cdot r^{(t)}\Bigl(\frac zt\Bigr)\>=\>
(tz-f_2)(z^2-4f_4)-f_3^2,$$
and we define
$$D_t\>:=\>\disc g_t(z)\ \in\R[x,y].$$
Using \eqref{discrvarsubst} and Lemma \ref{discrfr} we find
$$D_t\>=\>\disc g_t(z)\>=\>t^{-2}\,\disc f^{(t)}(z).$$
Explicitly, this gives
\begin{align*}
D_t & := -4f_2^3f_3^2-27t^2f_3^4+16f_2^4f_4-128t^2f_2^2f_4^2+144
  t^2f_2f_3^2f_4+256t^4f_4^3 \\
& \phantom:= 16f_4(4t^2f_4-f_2^2)^2+4f_2f_3^2(36t^2f_4-f_2^2)
  -27t^2f_3^4
\end{align*}
(a form of degree~$12$ in $x$ and $y$). We further put
\begin{equation}\label{dfnht}
h_t(z)\>=\>\frac\partial{\partial z}g_t(z)\>=\>3tz^2-2f_2z-4tf_4
\end{equation}
and conclude:
\end{lab}

\begin{lem}\label{dtinideal}
$D_t$ lies in the ideal generated by $g_t$ and $h_t$ in $\R[x,y,z]$.
\qed
\end{lem}


\section{Deforming the quartic, II: Case of a linear factor}%
\label{sec:defquii}

\begin{lab}
For $t\in\R$ we continue to consider the form
$$f^{(t)}\>=\>t^2z^4+f_2(x,y)z^2+f_3(x,y)z+f_4(x,y),$$
see \eqref{ft}. We know that $f^{(t)}$ is strictly positive definite
for $0<|t|\le1$, and that $f^{(t)}$ is a sum of three squares for
small $|t|$ (Prop.\ \ref{fortsezt0}).

Let $t_0\ne0$ be a fixed real number, and assume that the form
$f^{(t_0)}$ is strictly positive definite and a sum of three squares.
Under generic assumptions on $f$ \emph{which do not depend on $t_0$},
we shall show that $f^{(t)}$ is a sum of three squares for all $t$
sufficiently close to $t_0$.
\end{lab}

\begin{lab}
That $f^{(t_0)}$ is a sum of three squares means the following, by
\ref{ftso3s}: There exist forms $\xi_0$, $\eta_0\in\R[x,y]$ with
$\deg(\xi_0)=2$, $\deg(\eta_0)=3$ such that
\begin{equation}\label{t0gleichung}%
\eta_0^2\>=\>(f_2-t_0\xi_0)(4f_4-\xi_0^2)-f_3^2\>=\>g_{t_0}(\xi_0),
\end{equation}
and
\begin{equation}\label{t0psdbeding}%
f_2-t_0\xi_0\ge0,\quad4f_4-\xi_0^2\ge0.
\end{equation}
For $d\ge0$ we let again $V_d$ denote the vector space of forms of
degree $d$ in $\R[x,y]$. As in the proof of Prop.\ \ref{fortsezt0} we
consider the map $F\colon V_2\times V_3\times\R\to V_6$,
$$F(\xi,\eta,t)\>=\>\eta^2+f_3^2-(f_2-t\xi)(4f_4-\xi^2)\>=\>\eta^2
-g_t(\xi)$$
(see \ref{dfngtht} for $g_t(\xi)$).
The partial derivative of $F$ in $(\xi_0,\eta_0,t_0)$ with respect to
$(\xi,\eta)$ is the linear map
$$V_2\oplus V_3\to V_6,\quad(\xi,\eta)\>\mapsto\>2\eta_0\cdot\eta
-h_{t_0}(\xi_0)\cdot\xi$$
where
$$h_{t_0}(\xi_0)\>=\>3t_0\xi_0^2-2f_2\xi_0-4t_0f_4,$$
c.\,f.\ \ref{dfngtht}.
\end{lab}

\begin{prop}\label{extenduponasspt}
Assume that the two forms $\eta_0\in V_3$ and $h_{t_0}(\xi_0)\in V_4$
are relatively prime, and that $f_3\ne0$. Then there exist $\epsilon
>0$ and solutions $(\xi_t,\eta_t)$ to \eqref{xietaglngt1} and
\eqref{xietaglngt2} for $|t-t_0|<\epsilon$ such that $(\xi_{t_0},
\eta_{t_0})=(\xi_0,\eta_0)$.
\end{prop}

\begin{proof}
Indeed, by applying Lemma \ref{relprimlem} as in the proof of Prop.\
\ref{fortsezt0}, it follows from the theorem on implicit functions
that there are $(\xi_t,\eta_t)$ depending continuously (in fact
analytically, see \ref{contvsanal}) on $t$ and satisfying
$(\xi_{t_0},\eta_{t_0})=(\xi_0,\eta_0)$ and $F(\xi_t,\eta_t,t)=0$,
for $|t-t_0|<\epsilon'$ (with suitable $\epsilon'>0$). So equations
\eqref{xietaglngt1} hold for $|t-t_0|<\epsilon'$. We claim that
conditions \eqref{xietaglngt2} hold as well for suitable $0<\epsilon
\le\epsilon'$. Indeed, since $f_3\ne0$, this is clear if $f_2-t_0
\xi_0$ is strictly positive, see Remark \ref{rempsdconds}. If $f_2
-t_0\xi_0$ is a square, it could a~priori happen that the quadratic
form $f_2-t\xi_t$ is indefinite for $t$ arbitrarily close to $t_0$,
say with real zeros $\alpha_t<\beta_t$. However, since $(f_2-t\xi_t)
(4f_4-\xi_t^2)=f_3^2+\eta_t^2$, this would imply that $\alpha_t$ and
$\beta_t$ are roots of $f_3$ for all these $t$, which is evidently
impossible.
\end{proof}

\begin{lab}
It remains to show, \emph{under suitable generic assumptions on $f$},
that the following is true:
\begin{quote}
\emph{For every real number $t\ne0$ such that $f^{(t)}$ is positive
definite, and for every solution $(\xi,\eta)$ of \eqref{xietaglngt1}
and \eqref{xietaglngt2}, the two forms $\eta$ and $h_t(\xi)$ are
relatively prime.}
\end{quote}
To analyze the problem, assume that $\eta$ and $h_t(\xi)$ have a
nontrivial common divisor $p=p(x,y)$ in $\R[x,y]$. We can assume that
$p$ is irreducible, hence (homogeneous) of degree one or two. By
\eqref{xietaglngt1}, $\eta^2=g_t(\xi)$, and so $p$ divides $g_t(\xi)$
as well.

Below we will treat the case where $p$ is linear. The quadratic case
will be dealt with in Sect.~\ref{sec:defquiii}.
\end{lab}

\begin{lab}\label{1cas}%
So assume that $t\ne0$ and $f^{(t)}$ is positive definite, and $p$ is
a linear common divisor of $g_t(\xi)$ and $h_t(\xi)$ in $\R[x,y]$.
Let us denote equivalence in $\R[x,y]$ modulo the principal ideal
$(p)$ by~$\equiv$. By Lemma \ref{dtinideal}, $D_t$ lies in the ideal
generated by $g_t(\xi)$ and $h_t(\xi)$. We conclude that $D_t\equiv
0$.

Since $f^{(t)}$ is strictly positive definite, and since $\disc
f^{(t)}=t^2D_t$, Remark \ref{disc0posdef} implies $f_3\equiv4\,t^2f_4
-f_2^2\equiv0$. Since $p^2$ divides
$$g_t(\xi)\>=\>(f_2-t\xi)(4f_4-\xi^2)-f_3^2,$$
and since both factors $f_2-t\xi$ and $4f_4-\xi^2$ are psd, we
conclude that $p^2$ divides $f_2-t\xi$ or $4f_4-\xi^2$. From $t^2
(4f_4-\xi^2)=(f_2^2-t^2\xi^2)+(4t^2f_4-f_2^2)$ we see that in fact
$p^2$ divides $4f_4-\xi^2$ unconditionally, and that $p$ divides
$f_2^2-t^2\xi^2$. So we have
\begin{equation}\label{wasmuteilt}
\eta\>\equiv\>f_3\>\equiv\>f_2^2-t^2\xi^2\>\equiv\>0,\quad p^2\mid
(4f_4-\xi^2).
\end{equation}
From $f_2^2-t^2\xi^2\equiv0$ we see that one of the two conditions
$f_2\pm t\xi\equiv0$ holds. When $f_2-t\xi\equiv0$, this implies
$p^2\mid(f_2-t\xi)$
since $f_2-t\xi$ is psd, and so the right hand side of
$$\eta^2+f_3^2=(f_2-t\xi)(4f_4-\xi^2)$$
is divisible by $p^4$. This implies $p^2\mid f_3$, and so $f_3$
is not square-free, which is a non-generic situation. When $f_2+t\xi
\equiv0$, we combine this with $4f_4\equiv\xi^2$ to get
$$0\equiv h_t(\xi)=3t\xi^2-2f_2\xi-4tf_4\equiv(3+2-1)t\xi^2=
4t\xi^2.$$
This gives $\xi\equiv0$, and hence $f_4\equiv0$,
whence $(f_3,f_4)\ne1$. Again this is a non-generic situation.
\end{lab}


\section{Quadratic common divisors in pencils of polynomials}%
\label{sec:monster}

\begin{prop}\label{psithm}
Fix $m$, $n\ge2$ and consider triples $(f,g,h)$ of univariate
polynomials with $\deg(f)\le m$ and $\deg(g)$, $\deg(h)\le n$. There
exists a nonzero integral polynomial $\Psi_{m,n}(f,g,h)$ in the
coefficients of $f$, $g$, $h$ with the following property:

For any field $k$ and any polynomials $f$, $g$, $h\in k[x]$ with
$\deg(f)\le m$ and $\deg(g)$, $\deg(h)\le n$, if there exists $(0,0)
\ne(s,t)\in k^2$ with
$$\deg\gcd(f,\>sg+th)\>\ge\>2,$$
then $\Psi_{m,n}(f,g,h)=0$.
\end{prop}

\begin{proof}
Let $k$ be algebraically closed and $f\in k[x]$. Assume $\deg(f)=m$,
let $\alpha_1,\dots,\alpha_m$ be the roots of $f$, and assume that
the $\alpha_i$ are pairwise distinct, i.e., that $f$ is separable.
Given $g$ and $h$, there exists $(s,t)\ne(0,0)$ with $\deg\gcd(f,sg+
th)\ge2$ if and only if there exist $1\le i<j\le m$ such that
$$sg(\alpha_i)+th(\alpha_i)=sg(\alpha_j)+th(\alpha_j)=0$$
for some $(s,t)\ne(0,0)$, or equivalently, such that
$$g(\alpha_i)h(\alpha_j)=g(\alpha_j)h(\alpha_i).$$
So this holds if and only if
$$\tilde\phi(f,g,h)\>:=\>\prod_{1\le i<j\le m}\frac{g(\alpha_i)
h(\alpha_j)-g(\alpha_j)h(\alpha_i)}{\alpha_i-\alpha_j}$$
vanishes. It is easy to see that $\tilde\phi$ is invariant under all
permutations of the roots $\alpha_i$.
Hence when $f$ is monic, $\tilde\phi$ is an integral polynomial in
the coefficients of $f$, $g$ and $h$. To cover the non-monic case as
well, observe that $\tilde\phi$ has degree $\le(m-1)(n-1)$ with
respect to each $\alpha_i$. Therefore, if $a_0$ denotes the leading
coefficient of $f$, it follows that
$$\phi(f,g,h)\>:=\>a_0^{(m-1)(n-1)}\cdot\prod_{1\le i<j\le m}\frac
{g(\alpha_i)h(\alpha_j)-g(\alpha_j)h(\alpha_i)}{\alpha_i-\alpha_j}$$
is an integral polynomial in the coefficients of $f$, $g$ and $h$.

From $\phi(f,x,1)=1$ for monic $f$ of degree~$m$ we see that $\phi$
does not vanish identically. To prove the proposition it suffices to
put
$$\Psi_{m,n}(f,g,h)\>:=\>\disc_m(f)\cdot\phi(f,g,h).$$%
\end{proof}

\begin{dfn}
For polynomials $f$, $g$, $h\in k[x]$ with $\deg(f)\le m$ and $\deg
(g)$, $\deg(h)\le n$, we define the $\Phi$-invariant by
$$\Phi_{m,n}(f,g,h)\>:=\>a_0^{(m-1)(n-1)}\cdot\prod_{1\le i<j\le m}
\frac{g(\alpha_i)h(\alpha_j)-g(\alpha_j)h(\alpha_i)}{\alpha_i
-\alpha_j}$$
where $\alpha_1,\dots,\alpha_m$ are the roots of $f$ and $a_0$ is the
coefficient of $x^m$ in $f$. By the proof of Proposition
\ref{psithm}, $\Phi_{m,n}(f,g,h)$ is an integral polynomial in the
coefficients of $f$, $g$ and $h$.
\end{dfn}

The proof of Proposition \ref{psithm} has shown:

\begin{cor}\label{psithmcor}
In \ref{psithm} we can take
$$\Psi_{m,n}(f,g,h)\>=\>\disc_m(f)\cdot\Phi_{m,n}(f,g,h).$$
If $f$ is separable with $\deg(f)\ge m-1$, then $\Phi_{m,n}(f,g,h)=0$
is equivalent to the existence of a pair $(0,0)\ne(s,t)\in\ol k^2$
with $sg+th=0$ or $\deg\gcd(f,sg+th)\ge2$.
\qed
\end{cor}

\begin{rems}\label{remsphi}
1.\
The power of $a_0$ in the definition of $\Phi_{m,n}$ is the correct
one, in the sense that $\Phi_{m,n}$ is not divisible by $a_0$.
Indeed, if $f=\sum_{i=0}^ma_ix^{m-i}$, and if one takes $g:=x^{n-1}
(b_0x+b_1)$, $h:=x^{n-1}(c_0x+c_1)$, one finds
$$\Phi_{m,n}(f,g,h)\>=\>a_m^{(m-1)(n-1)}\cdot(b_0c_1-b_1c_0)^
{\choose m2}.$$

2.\
Write $f=\sum_{i=0}^ma_ix^{m-i}$, $g=\sum_{j=0}^nb_jx^{n-j}$ and
$h=\sum_{j=0}^nc_jx^{n-j}$. As a polynomial in the $a_i$, $b_j$ and
$c_j$, $\Phi_{m,n}$ is homogeneous of degree $(m-1)(n-1)$ in the
$a_i$ and of degree $\choose m2$ in the $b_j$ and in the $c_j$. If we
give degree $i$ to $a_i$ and degree $j$ to $b_j$ and $c_j$, then
$\Phi_{m,n}$ is jointly homogeneous in all variables of degree
$\choose m2(2n-1)$.

3.\
The $\Phi$-invariant has some relations with resultants. For example,
the rule
$$\Phi_{m,n+d}(f,pg,ph)\>=\>\res_{m,d}(f,p)^{m-1}\cdot\Phi_{m,n}
(f,g,h)$$
holds, for $\deg(f)\le m$, $\deg(p)\le d$ and $\deg(g)$, $\deg(h)\le
n$.
\end{rems}

\begin{example}
Let $a_i$, $b_j$, $c_j$ be the coefficients of $f$, $g$, $h$ as
before. In low degrees it is quite manageable to calculate $\Phi$
explicitly. For example we have
$$\Phi_{2,2}(f,g,h)=\det(f,g,h)=\left|\begin{array}{ccc}a_0&b_0&c_0\\
a_1&b_1&c_1\\a_2&b_2&c_2\end{array}\right|$$
or
\begin{align*}
\Phi_{2,3}(f,g,h) & = a_0^2b_2c_3-a_0a_2b_2c_1+a_1a_2b_2c_0-a_0a_1b_1
  c_3+a_1^2b_0c_3 \\
& -a_0a_2b_0c_3+a_2^2b_0c_1-a_0^2b_3c_2+a_0a_2b_1c_2-a_1a_2b_0c_2 \\
& +a_0a_1b_3c_1-a_1^2b_3c_0+a_0a_2b_3c_0-a_2^2b_1c_0.
\end{align*}
As the remarks on the degree of $\Phi_{m,n}$ show, the size of $\Phi_
{m,n}$ grows quickly with $m$ and $n$.
\end{example}

We do not know whether $\Phi_{m,n}(f,g,h)$ or some related invariant
has been considered before.


\section{Deforming the quadric, III: Case of a quadratic factor}%
\label{sec:defquiii}

As before we write
$$g_t(\xi)\>=\>t\xi^3-f_2\xi^2-4tf_4\xi+f_6$$
where $f_6:=4f_2f_4-f_3^2$, and
$$h_t(\xi)\>=\>\frac\partial{\partial\xi}g_t(\xi)\>=\>3t\xi^2-2f_2
\xi-4tf_4.$$
The hardest step in our proof is to show, for generically chosen
$f_i$, that $g_t(\xi)$ and $h_t(\xi)$ have no common quadratic
factor, whenever $(\xi,\eta)$ is a solution of \eqref{xietaglngt1}
and $t\ne0$. This will be accomplished by the following result:

\begin{prop}\label{kopfweh}
Consider triples $(f_2,f_3,f_4)$ of forms in $\R[x,y]$ (with $\deg
(f_i)=i$ for $i=2,3,4$) for which
\begin{equation}\label{xietaeqnt}
\eta^2\>=\>(f_2-t\xi)(4f_4-\xi^2)-f_3^2\>=\>g_t(\xi)
\end{equation}
has a solution $(\xi,\eta)$ for some $0\ne t\in\R$ such that $g_t
(\xi)$ and $h_t(\xi)$ have a common irreducible quadratic factor.
Then these triples are not Zariski dense.
\end{prop}

In other words, there exists a nonzero polynomial $\Psi=\Psi(f_2,f_3,
f_4)$ in the coefficients of $f_2$, $f_3$ and $f_4$ which vanishes on
the triples described in the proposition.

The plan of the proof is as follows. We will successively deduce six
``exceptional'' conditions on $(f_2,f_3,f_4)$, labelled
$(S_1)$--$(S_6)$. We will show that, for generic choice of the $f_i$,
none of these conditions holds. On the other hand, we'll show that
the assumptions of \ref{kopfweh} imply that at least one of
$(S_1)$--$(S_6)$ is satisfied.

\begin{lab}
We dehomogenize all forms in $\R[x,y]$ by setting $y=1$. So $f_2$,
$f_3$, $f_4$, $\xi$, $\eta$ are polynomials in $\R[x]$ with $\deg
(f_i)\le i$ ($i=2,3,4$), $\deg(\xi)\le2$ and $\deg(\eta)\le3$. We
assume that $t\ne0$ is a real number and identity \eqref{xietaeqnt}
holds, and that $p\in\R[x]$ is an irreducible quadratic polynomial
with $p^2\mid g_t(\xi)$ and $p\mid h_t(\xi)$. Denoting congruences
modulo~$(p)$ in $\R[x]$ by $\equiv$, we therefore have
\begin{equation}\label{eq1}
t\xi^3-f_2\xi^2-4tf_4\xi+f_6\>=\>(f_2-t\xi)(4f_4-\xi^2)-f_3^2\>\equiv
\>0
\end{equation}
and
\begin{equation}\label{eq2}
3t\xi^2-2f_2\xi-4tf_4\>\equiv\>0.
\end{equation}
Combining \eqref{eq1} and \eqref{eq2} we get
\begin{equation}\label{eq1p}
f_2\xi^2+8tf_4\xi-3f_6\>\equiv\>0,
\end{equation}
and eliminating $t$ from \eqref{eq2} and \eqref{eq1p} we find
\begin{equation}\label{eq4}
f_2\xi^4-(8f_2f_4-3f_3^2)\xi^2+4f_4f_6\>\equiv\>0.
\end{equation}
\end{lab}

\begin{lab}
We use $'$ to denote the derivative $\frac d{dx}$ on polynomials in
$\R[x]$. From $p^2\mid g_t(\xi)$ we see that $p$ divides $g_t(\xi)'=
h_t(\xi)\xi'-(f_2'\xi^2+4tf_4'\xi-f_6')$, and hence
\begin{equation}\label{eq3}
f_2'\xi^2+4tf_4'\xi-f_6'\>\equiv\>0.
\end{equation}
From $h_t(\xi)\equiv0$ and \eqref{eq3} we can again eliminate $t$ and
get
\begin{equation}\label{eq5}
3f_2'\xi^4+(8f_2f_4'-4f_2'f_4-3f_6')\xi^2+4f_4f_6'\>\equiv\>0.
\end{equation}
\end{lab}

\begin{lab}
For $i,\>j\in\{2,3,4\}$ we put
$$g_{ij}\>:=\>if_if_j'-jf_jf_i'\>=\>f_if_j\,\frac d{dx}\,\log(f_j^i
f_i^{-j}).$$
Note that $\deg(g_{ij})\le i+j-2$, with equality for generic choice
of the $f_i$. We observe the relation
$$2f_2g_{34}-3f_3g_{24}+4f_4g_{23}\>=\>0.$$
\end{lab}

\begin{lab}
We now eliminate $\xi$. From \eqref{eq4} and \eqref{eq5} we can
eliminate $\xi^4$ and get
\begin{equation}\label{eq6}
(2f_2g_{24}-3f_3g_{23})\xi^2-4f_4(2f_2g_{24}-f_3g_{23})\>\equiv\>0.
\end{equation}
We can eliminate $t$ from \eqref{eq1p} and \eqref{eq3}, getting
\begin{equation}\label{eq7}
g_{24}\xi^2+2(f_3g_{34}-2f_4g_{24})\>\equiv\>0.
\end{equation}
Finally we can eliminate $\xi$ from \eqref{eq6} and \eqref{eq7},
getting
\begin{equation}\label{eq8}
f_3^2\cdot\bigl(g_{23}g_{34}-g_{24}^2\bigr)\>\equiv\>0.
\end{equation}
\end{lab}

\begin{lab}\label{sonder1}
We introduce the following ``exceptional'' conditions (S1)--(S3).
Clearly, none of them holds for generically chosen $f_2$, $f_3$,
$f_4$:
\begin{itemize}
\item[$(S_1)$]
$\gcd(f_3,f_4)\ne1$,
\item[$(S_2)$]
$\gcd(g_{23},\,g_{24})\ne1$,
\item[$(S_3)$]
$\gcd(g_{34},\,g_{24})\ne1$.
\end{itemize}
\end{lab}

\begin{lab}\label{f3eq0}
We show that $f_3\equiv0$ leads to an exceptional case. Assume that
$(S_1)$ is excluded and $f_3\equiv0$. From \eqref{eq1} we get
$(f_2-t\xi)(4f_4-\xi^2)\equiv0$, hence
$$f_2-t\xi\>\equiv\>0\quad\text{or}\quad4f_4-\xi^2\>\equiv\>0.$$
$f_2\equiv t\xi$, together with \eqref{eq2}, gives $4f_4\equiv\xi^2$
since $t\ne0$. Conversely, $4f_4\equiv\xi^2$ and \eqref{eq1p} imply
$f_4(f_2-t\xi)\equiv0$,
and $f_4\not\equiv0$ since $\gcd(f_3,f_4)=1$. So we see that $f_3
\equiv f_2-t\xi\equiv4f_4-\xi^2\equiv\>0$ hold in any case, and
therefore also $f_3\equiv f_2^2-4t^2f_4\equiv0$.
In particular, there exists a scalar $\lambda$ such that $\deg
\gcd(f_3,\,f_2^2+\lambda f_4)\ge2$. By Proposition \ref{psithm} this
means we are in the following exceptional case:
\begin{itemize}
\item[$(S_4)$]
$\Psi_{3,4}(f_3,\,f_2^2,\,f_4)=0$.
\end{itemize}
\end{lab}

\begin{lab}\label{g24}
Excluding $(S_1)$ and $(S_4)$ we have $f_3\not\equiv0$, and therefore
get
\begin{equation}\label{eq9}
g_{23}g_{34}-g_{24}^2\>\equiv\>0
\end{equation}
from \eqref{eq8}. The assumption $g_{24}\equiv0$ leads to one of
$(S_2)$ or $(S_3)$. Excluding those we have in addition $g_{24}\not
\equiv0$.
\end{lab}

\begin{lab}
We finally assume that $(S_1)$--$(S_4)$ are excluded, so \eqref{eq9}
holds with $g_{24}\not\equiv0$. We show that this again leads to an
exceptional case. Multiply \eqref{eq7} with $g_{23}$, rewrite using
\eqref{eq9} and cancel the factor $g_{24}$ to get
\begin{equation}\label{eq7p}
g_{23}\xi^2+2f_3g_{24}-4f_4g_{23}\>\equiv\>0.
\end{equation}
Multiply \eqref{eq1p} with $g_{23}$ and use \eqref{eq7p} to obtain
$$8tf_4g_{23}\xi\>\equiv\>(8f_2f_4-3f_3^2)g_{23}+2f_2f_3g_{24}.$$
Squaring this congruence and using \eqref{eq7p} once more we finally
get
\begin{equation}\label{eq10}
128\,t^2f_4^2g_{23}(2f_4g_{23}-f_3g_{24})-\Bigl((8f_2f_4-3f_3^2)
g_{23}+2f_2f_3g_{24}\Bigr)^2\>\equiv\>0.
\end{equation}
\end{lab}

\begin{lab}
Consider
\begin{align*}
P & := g_{23}g_{34}-g_{24}^2, \\
Q & := f_4^2g_{23}\,(2f_4g_{23}-f_3g_{24}), \\
R & := (8f_2f_4-3f_3^2)g_{23}+2f_2f_3g_{24}.
\end{align*}
These are integral polynomials in the coefficients of $f_2$, $f_3$,
$f_4$. For generically chosen $f_i$ we have $\deg(P)=8$, $\deg(Q)=18$
and $\deg(R)=9$. We have shown that the assumption in Proposition
\ref{kopfweh} leads either to one of $(S_1)$--$(S_4)$, or to the
existence of a pair $(\lambda,\mu)\ne(0,0)$ of scalars with $\deg\gcd
(P,\,\lambda Q+\mu R^2)\ge2$. By Proposition \ref{psithm} and
Corollary \ref{psithmcor}, the latter implies one of the following
two conditions:
\begin{itemize}
\item[$(S_5)$]
$\disc_8(P)=0$;
\item[$(S_6)$]
$\Phi_{8,18}(P,\,Q,\,R^2)=0$.
\end{itemize}
\end{lab}

\begin{lab}\label{psipqr2ne0}
We still need to show that
$$\Psi_{8,18}(P,Q,R^2)\>=\>\disc_8(P)\cdot\Phi_{8,18}(P,Q,R^2)\>\ne\>
0$$
for generically chosen $f_i$. Clearly it suffices to exhibit a single
triple $(f_2,f_3,f_4)$ where this number is nonzero. Unfortunately,
it seems hard to do this by hand alone, due to the enormous size of
the polynomial $\Phi$. With the help of a computer algebra program,
there is no difficulty: If we take
$$f_2\>=\>x^2-x+1,\quad f_3\>=\>x^2-1,\quad f_4\>=\>x^4+1,$$
then
{\small
$$P=g_{23}g_{34}-g_{24}^2=-24x^8+60x^7-64x^5+56x^4-20x^3-144x^2+88x
-16$$}%
is separable,
and $\Phi_{8,18}(P,Q,R^2)$ is an integer with $372$ digits that has
the prime factorization
{\small
$$-\,2^{713}\cdot3^{33}\cdot179\cdot233\cdot641\cdot1531\cdot4093
\cdot 11273\cdot29983^7\cdot342841^{14}\cdot66617977107707$$}%
\end{lab}

\begin{rem}
We can consider $\Phi_{8,18}(P,Q,R^2)$ as an integral polynomial in
the coefficients of $f_2$, $f_3$, $f_4$. To find the degree of this
polynomial, note that $\Phi_{8,18}(P,Q,R^2)$ is homogeneous of degree
$7\cdot17=119$ in the coefficients of $P$, and homogeneous of degree
$\choose82=28$ in the coefficients of $Q$ and in those of $R^2$
(Remark \ref{remsphi}.2). Given that $P$ (resp.\ $Q$, resp.\ $R^2$)
is homogeneous of degree~$4$ (resp.~$7$, resp.~$8$) in $(f_2,f_3,
f_4)$, we conclude that $\Phi_{8,18}(P,Q,R^2)$ is homogeneous of
degree
$$119\cdot4+28\cdot7+28\cdot8\>=\>896$$
in (the coefficients of) $f_2$, $f_3$ and $f_4$.
\end{rem}

\begin{rem}
The invariant $\Phi_{8,18}(P,Q,R^2)$ is enormous not only by its
degree, but also in terms of the values it produces. If the $f_i$
have small integral coefficients, then $\Phi(P,Q,R^2)$ typically has
several hundreds of digits.

Based on the factorization of this invariant in several sample cases
with integer coefficients, we suspect that the form $\Phi(P,Q,R^2)$
(of degree $896$ in the coefficients of $f_2$, $f_3$ and $f_4$)
decomposes as a product of smaller degree forms.
\end{rem}


\section{Summary and complements}\label{sec:summ}

\begin{lab}\label{exclist}
Let
\begin{equation}\label{grundf}
f\>=\>z^4+f_2z^2+f_3z+f_4
\end{equation}
where $f_i\in\R[x,y]$ is homogeneous of degree $i$ ($i=2,3,4$) and
$f_2$, $f_4$, $4f_2f_4-f_3^2$ are psd. In the course of our proof of
Hilbert's theorem we have considered the following exceptional cases:
\begin{quote}
\begin{itemize}
\item[$(E_1)$]
$\disc_2(f_2)=0$,
\item[$(E_2)$]
$f_2\mid f_3$,
\item[$(E_3)$]
$\disc_6(4f_2f_4-f_3^2)=0$,
\item[$(E_4)$]
$\disc_3(f_3)=0$,
\item[$(E_5)$]
$\gcd(f_3,f_4)\ne1$,
\item[$(E_6)$]
$\gcd(g_{23},g_{24})\ne1$,
\item[$(E_7)$]
$\gcd(g_{24},g_{34})\ne1$,
\item[$(E_8)$]
$\Phi_{3,4}(f_3,f_2^2,f_4)=0$,
\item[$(E_9)$]
$\disc_8(g_{23}g_{34}-g_{24}^2)=0$,
\item[$(E_{10})$]
$\Phi_{8,18}(P,Q,R^2)=0$.
\end{itemize}
\end{quote}
(Note that the conditions $\gcd\ne1$ can be rephrased as the
vanishing of suitable resultants.) For counting inequivalent
representations, we also needed to consider the following condition:
\begin{quote}
\begin{itemize}
\item[$(E_{11})$]
$\gcd(f_3,\,4f_4-f_2^2)\ne1$.
\end{itemize}
\end{quote}
Let us summarize the role of these exceptional cases. For every real
number $t$ we considered the equation
$$C_t:\ \eta^2+f_3^2\>=\>(f_2-t\xi)(4f_4-\xi^2)$$
with the side conditions $f_2-t\xi\ge0$ and $4f_4-\xi^2\ge0$.

We had to exclude $(E_1)$ to ensure that $C_0$ has a solution
$(\xi_0,\eta_0)$ (for $t=0$, Cor.\ \ref{pfproph}).

We had to exclude $(E_2)$, $(E_3)$ to extend any solution of $C_0$
to a solution of $C_t$ for small $|t|$ (Prop.\ \ref{fortsezt0}).

We had to exclude $f_3=0$ (which is contained in $(E_4)$), and had to
assume $\gcd(g_t(\xi),\,h_t(\xi))=1$ for all $0<|t|<1$ and all
solutions $(\xi,\eta)$ of $C_t$, to extend a solution of $C_t$
for $0<|t|<1$ into a neighborhood of $t$ (see \ref{extenduponasspt}).

We had to exclude $(E_4)$ and $(E_5)$ to exclude a linear common
divisor of $g_t(\xi)$ and $h_t(\xi)$ (see \ref{1cas}).

We had to exclude $(E_5)$--$(E_{10})$ to exclude an irreducible
quadratic common divisor of $g_t(\xi)$ and $h_t(\xi)$ (see Sect.\
\ref{sec:defquiii}, these conditions were labelled $(S_1)$--$(S_6)$
there).
\end{lab}

\begin{lab}
We have proved: If the quartic form
\begin{equation*}\label{normiertf}%
f\>=\>z^4+f_2z^2+f_3z+f_4
\end{equation*}
is strictly positive definite with $f-z^4\ge0$, and if $f$ is
sufficiently generic, then any solution of $C_0$ (for $t=0$) can be
extended in a unique continuous way to a solution of $C_t$, for $0\le
t\le1$. Here ``sufficiently generic'' means that $f$ avoids the
exceptional cases $(E_1)$--$(E_{10})$. For $i=1,\dots,10$, there
exists a nonzero polynomial $\Psi_i$ in (the coefficients of) $f$
such that $\Psi_i(f)\ne0$ if and only if $f$ avoids $(E_i)$. Clearly,
the set of strictly positive definite forms $f$ with $\prod_{i=1}^
{10}\Psi_i(f)\ne0$ is dense in the space of all psd forms of shape
\eqref{normiertf}. By \ref{limarg}, it follows that any psd form
\eqref{normiertf} is a sum of three squares.
\end{lab}

\begin{example}
An explicit example of a positive definite form $f$ which is
``sufficiently generic'' is
$$f\>=\>z^4+(x^2-xy+y^2)z^2+(x^2-y^2)yz+(x^4+y^4).$$
That is, $f$ avoids all exceptional conditions $(E_1)$--$(E_{11})$.
(See \ref{psipqr2ne0} for $(E_9)$ and $(E_{10})$; the other
conditions are readily checked except possibly $(E_8)$, which is
avoided since $\Phi_{3,4}(f_3,f_2^2,f_4)=56$.)
\end{example}

\begin{lab}
Along our proof of Hilbert's theorem, we needed only little extra
effort to obtain partial information on the number of inequivalent
representations of a psd form $f$ as a sum of three squares. (See
Definition \ref{dfnortheq} for the meaning of equivalence of
representations.) Let us review and complete these results:
\end{lab}

\begin{thm}
Let $f$ be a psd form.
\begin{itemize}
\item[(a)]
When $f$ has a real zero and is otherwise sufficiently generic,
then $f$ has precisely $4$ inequivalent representations.
\item[(b)]
When $f$ is strictly positive and sufficiently generic, then $f$ has
precisely $8$ inequivalent representations.
\end{itemize}
Here, ``sufficiently generic'' means in (a) that $f$ avoids
$(E_1)$--$(E_3)$, assuming $f(0,0,1)=0$. In (b) it means that $f$
avoids $(E_1)$--$(E_{11})$ if $f$ is normalized into the form $f=
z^4+f_2z^2+f_3z+f_4$ with $f-z^4\ge0$.
\end{thm}

\begin{proof}
(a) was proved in Prop.\ \ref{summreal0}. For the proof of (b)
assume that $f$ is normalized as above (Lemma \ref{normalisform}),
and consider the linear pencil $f^{(t)}$ as in \eqref{ft}. When $f$
avoids $(E_1)$--$(E_{10})$, we have proved that we can extend every
solution $(\xi_0,\eta_0)$ of $C_0$ (at time $t=0$) along this
pencil to a solution $(\xi,\eta)$ of $C_1$ (at time $t=1$), and
that locally this extension is everywhere unique. Hence, for $t=1$
there are at least as many solutions $(\xi,\eta)$ as for $t=0$,
namely $16$ (see \ref{gener164}). If we also exclude $(E_{11})$,
then Corollary \ref{eindtfall2} shows that for $f$ (i.e., for $t=1$)
these $16$ pairs $(\xi,\eta)$ correspond to precisely $8$
inequivalent representations. In order to show that there are no
further representations of $f$, we need to show that for $t\to0$ the
solutions $(\xi^{(t)},\,\eta^{(t)})$ of $C_t$ remain bounded, and
thus converge to solutions for $t=0$. But this is obvious since we
have $4f_4-(\xi^{(t)})^2\ge0$ for all $t$.
\end{proof}

\begin{rem}
These findings are in agreement with the results of \cite{PRSS} and
\cite{sch:jag}. As far as we know, this is the first time that
results on the number of inequivalent representations have been
obtained by elementary methods.
\end{rem}


\end{document}